\documentclass{amsart}
\usepackage{amsfonts}

\newtheorem{theorem}{Theorem}
\theoremstyle{plain}

\newtheorem{case}{Case}

\newtheorem{corollary}{Corollary}

\newtheorem{remark}{Remark}

\numberwithin{equation}{section}
\input{tcilatex}

\begin{document}
\title[polynomial identities]{Polynomial identities that involve binomial
coefficients, double and rising factorials and their probabilistic
interpretations and proofs.}
\author{Pawe\l\ J. Szab\l owski}
\address{Emeritus in Department of Mathematics and Information Sciences, \\
Warsaw University of Technology\\
ul Koszykowa 75, 00-662 Warsaw, Poland \\
}
\email{pawel.szablowski@gmail.com}
\urladdr{}
\date{May 2022}
\thanks{}

\begin{abstract}
We formulate several polynomial identities. One side of these identities has
a nice simple form. Whereas the other has a form of a polynomial whose
coefficients contain binomial coefficients double factorials or (and) rising
factorials. The origins and the proofs of these identities are
probabilistic. However, their form suggests universal applications in
simplifying expressions. Many useful simplifying formulae are presented in
the sequel.
\end{abstract}

\maketitle

\section{Introduction}

In this note, we will present some identities involving factorials, binomial
coefficients, and the so-called rising factorials (sometimes called
Pochhammer symbols). Many of them have the forms of polynomials whose
coefficients often have the form of binomial coefficients and (or) of rising
factorials of some additional variables. These variables appear on both
sides of the identity. The domains of these variables can be extended to all
complex numbers. All these identities have a common origin, except for the
sometimes similar form. Namely, the origins of all are a few, rather deep
probabilistic interpretations and sometimes following non-trivial
computations.

The paper is organized as follows. The next section is dedicated to the
presentation of the identities and some of the particular, interesting
particular cases. The next section is devoted to the presentation of the
probabilistic background of the results presented in the previous section
and then, finally, the presentation of the calculations leading to the
identities.

\section{Identities}

Let us set $(-1)!!\allowbreak =\allowbreak (0)!!\allowbreak =\allowbreak 1$
and $\binom{n}{k}\allowbreak =\allowbreak 0$ when $n<k$. $i$ would always
denote imaginary unit , i.e., $i\allowbreak =\allowbreak \sqrt{-1}%
\allowbreak =\allowbreak \exp (i\pi /2)$. In order to simplify notation let
us introduce the so-called rising factorial (sometimes called also
Pochhammer symbol), which is the following function:%
\begin{equation*}
(x)^{(n)}=x(x+1)\ldots (x+n-1),
\end{equation*}%
defined for all complex $x.$ Notice that we have for all $x\neq 0:$%
\begin{equation*}
(x)^{(n)}=\frac{\Gamma (x+n)}{\Gamma (x)},
\end{equation*}%
where $\Gamma (x)$ denotes the Euler's gamma function. To learn more about
Pochhammer symbols or binomial coefficients, see, e.g., \cite{Grad14}.
Notice only that, e.g., $\left( 1\right) ^{\left( n\right) }\allowbreak
=\allowbreak n!$ or $\left( 1/2\right) ^{\left( n\right) }\allowbreak
=\allowbreak (2n-1)!!/2^{n}$. We will use these values below.

\begin{theorem}
\label{Main}For all natural $k$ and all complex $\rho $ and $\beta $ we have:

i)%
\begin{equation}
\frac{(2k)!}{k!}(1+\rho )^{k}=\sum_{j=0}^{k}\binom{2k}{2j}(1+\rho
)^{2k-2j}(1-\rho ^{2})^{j}(2j-1)!!(2k-2j-1)!!.  \label{G1}
\end{equation}

ii)%
\begin{equation}
(1+\rho )^{k}=\frac{1}{2^{k}}\sum_{j=0}^{2k}\sum_{m=0}^{\left\lfloor
j/2\right\rfloor }(2\rho )^{j-2m}\binom{k}{j-2m}\binom{k-j+2m}{m}.
\label{G2}
\end{equation}

iii) For all integer $m$ we have 
\begin{equation}
\sum_{j=0}^{n}(-1)^{j}\binom{n}{j}\frac{\left( \beta \right) ^{\left(
j+m\right) }}{\left( \beta \right) ^{\left( j\right) }}=(-1)^{n}n!\binom{m}{n%
}\frac{\left( \beta \right) ^{\left( m\right) }}{\left( \beta \right)
^{\left( n\right) }}.  \label{Ex1}
\end{equation}

iv)%
\begin{gather}
\sum_{m=0}^{n}\left( -1\right) ^{m}\binom{n}{m}\left( \beta \right) ^{\left(
m\right) }\left( \beta \right) ^{\left( n-m\right) }\sum_{k=0}^{m}\binom{m}{k%
}\binom{n-m}{k}\frac{\rho ^{k}k!}{\left( \beta \right) ^{\left( k\right) }}
\label{Ex2} \\
=\left\{ 
\begin{array}{ccc}
0 & \text{if} & n\text{ is odd} \\ 
\frac{n!}{(n/2)!}\left( \beta \right) ^{\left( n/2\right) }(1-\rho )^{n/2} & 
\text{if} & n\text{ is even}%
\end{array}%
\right. .  \notag
\end{gather}

v) 
\begin{eqnarray}
&&\sum_{m=0}^{n}(-1)^{n-m}\binom{n}{m}\sum_{j=0}^{m}\binom{m}{j}(1-\rho
)^{j}\rho ^{m-j}\left( \beta \right) ^{\left( n-j\right) }\left( \beta
+m-j\right) ^{\left( j\right) }  \label{Ex3} \\
&=&\left\{ 
\begin{array}{ccc}
0 & \text{if} & n\text{ is odd} \\ 
\frac{n!}{(n/2)!}\left( \beta \right) ^{\left( n/2\right) }(1-\rho )^{n/2} & 
\text{if} & n\text{ is even}%
\end{array}%
\right. .  \notag
\end{eqnarray}
\end{theorem}

Below we present some, believed to be important, particular cases, remarks
and corollaries.

\begin{remark}
Taking first $\rho \allowbreak =\allowbreak \frac{2}{3}$ then $\rho
\allowbreak =\allowbreak \frac{1}{3}$ and finally $\rho \allowbreak
=\allowbreak \frac{4}{3}$ in (\ref{G1}) we get: 
\begin{eqnarray*}
\sum_{j=0}^{k}\binom{2k}{2j}5^{j}(2j-1)!!(2k-2j-1)!! &=&\frac{(2k)!3^{k}}{k!}%
, \\
\sum_{j=0}^{k}\binom{2k}{2j}2^{j}(2j-1)!!(2k-2j-1)!! &=&\frac{(2k)!3^{k}}{%
k!2^{k}}, \\
\sum_{j=0}^{k}\binom{2k}{2j}(-7)^{j}(2j-1)!!(2k-2j-1)!! &=&(-1)^{k}\frac{%
(2k)!3^{k}}{k!}.
\end{eqnarray*}
\end{remark}

\begin{remark}
Let us take $\rho \allowbreak =\allowbreak \sqrt{5}$ in (\ref{G1}) and
recall that $(1+\sqrt{5})^{n}\allowbreak =\allowbreak 2^{n}(\frac{1+\sqrt{5}%
}{2})^{n}\allowbreak =\allowbreak 2^{n-1}(L_{n}+\sqrt{5}F_{n}),$ where $%
L_{n} $ and $F_{n}$ are respectively $n-$th Lucas and Fibonacci numbers.
Now, equating integer and irrational parts we end up with the following two
identities valid for $k\geq 0:$%
\begin{eqnarray*}
\frac{(2k)!}{k!}L_{k} &=&2^{k}\sum_{j=0}^{k}\binom{2k}{2j}%
(-1)^{j}L_{2k-2j}(2j-1)!!(2k-2j-1)!!, \\
\frac{(2k)!}{k!}F_{k} &=&2^{k}\sum_{j=0}^{k}\binom{2k}{2j}%
(-1)^{j}F_{2k-2j}(2j-1)!!(2k-2j-1)!!
\end{eqnarray*}
\end{remark}

\begin{remark}
Taking $\rho \allowbreak =\allowbreak i$ in (\ref{G1}), further noticing
that $(1+i)^{k}\allowbreak =\allowbreak 2^{k/2}\exp (ik\pi /4)$ and taking
firstly $k\allowbreak =\allowbreak 4n$ and cancelling out $(-4)^{n}$ on both
sides, we get for all $n\geq 0:$%
\begin{equation*}
\frac{(8n)!}{(4n)!}=\sum_{j=0}^{4n}\binom{8n}{2j}(2j-1)!!(8n-2j-1)!!,
\end{equation*}%
and the taking $k\allowbreak =\allowbreak 4n+1$ and equating the real and
imaginary parts we get for all $n\geq 0:$%
\begin{eqnarray*}
\frac{(8n+2)!}{(4n+1)!}(-4)^{n} &=&2^{4n+1}\sum_{m=0}^{2n}(-1)^{m}\binom{8n+2%
}{4m+2}(4m+1)!!(8n-4m-1)!!, \\
\frac{(8n+2)!}{(4n+1)!}(-4)^{n} &=&2^{4n+1}\sum_{m=0}^{2n}(-1)^{m}\binom{8n+2%
}{4m}(4m-1)!!(8n-4m+1)!!.
\end{eqnarray*}
\end{remark}

\begin{remark}
Taking $\rho \allowbreak =\allowbreak 1/2$ in (\ref{G2}) we get the
following identity:%
\begin{equation*}
1=\sum_{j=0}^{2k}(-1)^{j}\sum_{m=0}^{\left\lfloor j/2\right\rfloor }\binom{k%
}{j-2m}\binom{k-j+2m}{m}.
\end{equation*}
\end{remark}

\begin{remark}
Taking $\rho \allowbreak =\allowbreak i$ and $k\allowbreak =\allowbreak 4n$
in (\ref{G2}), further noticing that $(1+i)^{4n}\allowbreak =\allowbreak
(-4)^{n}$ and $(2i)^{j-2m}\allowbreak =\allowbreak i^{j}(-1)^{m}$ and
finally splitting the right-hand side into real and imaginary parts we get
for all $n\geq 0:$%
\begin{eqnarray*}
4^{n} &=&\frac{(-1)^{n}}{16^{n}}\sum_{j=0}^{4n}(-1)^{j}%
\sum_{m=0}^{j}(-1)^{m}4^{j-m}\binom{4n}{2j-2m}\binom{4n-2j+2m}{m}, \\
0 &=&\sum_{j=0}^{4n}(-1)^{j}\sum_{m=0}^{j}(-1)^{m}4^{j-m}\binom{4n}{2j-2m+1}%
\binom{4n-2j-1+2m}{m}.
\end{eqnarray*}
\end{remark}

\begin{remark}
i) Taking $\beta \allowbreak =\allowbreak 1$ and integer $k$ in (\ref{Ex1})
we get%
\begin{equation*}
\sum_{j=0}^{n}(-1)^{j}\binom{n}{j}\binom{j+k}{j}=(-1)^{n}\binom{k}{n}.
\end{equation*}

ii) Taking $\beta \allowbreak =\allowbreak 1/2$ and integer $k$ in (\ref{Ex1}%
) we get:%
\begin{equation*}
\sum_{j=0}^{n}(-1)^{j}\frac{(2n-1)!!(2j+2k-1)!!}{j!(n-j)!(2j-1)!!}%
=(-1)^{n}2^{n}\binom{k}{n}(2k-1)!!.
\end{equation*}
\end{remark}

\begin{remark}
i) Taking $\beta \allowbreak =\allowbreak 1/2$ and and integer $k$ in (\ref%
{Ex2}) we get:%
\begin{equation*}
\sum_{j=0}^{n}(-1)^{n-j}\binom{n}{j}\frac{(2n+2k+2j-1)!!}{(2j-1)!!}=2^{n}%
\frac{(n+k)!(2n+2k-1)!!}{k!(2n-1)!!}.
\end{equation*}

ii) Taking $\beta \allowbreak =\allowbreak 1$ in (\ref{Ex2}) we get for all
integer $n$ and $k.$ 
\begin{equation}
\sum_{m=0}^{n}\left( -1\right) ^{m}\sum_{k=0}^{m}\binom{m}{k}\binom{n-m}{k}%
\rho ^{k}=\left\{ 
\begin{array}{ccc}
0 & \text{if} & n\text{ is odd} \\ 
(1-\rho )^{n/2} & \text{if} & n\text{ is even}%
\end{array}%
\right. .  \label{explr}
\end{equation}

iii) Changing the order of summation in (\ref{explr}) and comparing the
coefficients in expansions in powers of $\rho $ we get:%
\begin{equation*}
\sum_{m=k}^{n}(-1)^{m-k}\binom{m}{k}\binom{n-m}{k}=\left\{ 
\begin{array}{ccc}
0 & \text{if} & n\text{ is odd} \\ 
\binom{n/2}{k} & \text{if} & n\text{ is even}%
\end{array}%
\right. .
\end{equation*}%
iv) Taking additionally $\rho \allowbreak =\allowbreak 1$ in we get, for
all, $n\geq 0:$ 
\begin{equation*}
\sum_{m=0}^{n}\left( -1\right) ^{m}\sum_{k=0}^{m}\binom{m}{k}\binom{n-m}{k}%
=0.
\end{equation*}

v) Taking $\beta \allowbreak =\allowbreak 1/2$ in (\ref{Ex2}) we get after
denoting $2\rho \allowbreak =\allowbreak x$ and canceling out $n!:$%
\begin{gather*}
\sum_{m=0}^{n}\left( -1\right) ^{m}(2m-1)!!(2n-2m-1)!!\sum_{k=0}^{m}\frac{1}{%
k!(m-k)!(n-m-k)!}\frac{x^{k}}{(2k-1)!!} \\
=\left\{ 
\begin{array}{ccc}
0 & \text{if} & n\text{ is odd} \\ 
\frac{(n-1)!!}{(n/2)!}(2-x)^{n/2} & \text{if} & n\text{ is even}%
\end{array}%
\right. .
\end{gather*}%
vi) By taking additionally in the above-mentioned identity $x\allowbreak
=\allowbreak 0$ we get:%
\begin{equation*}
\sum_{m=0}^{n}\left( -1\right) ^{m}\binom{n}{m}(2m-1)!!(2n-2m-1)!!=\left\{ 
\begin{array}{ccc}
0 & \text{if} & n\text{ is odd} \\ 
\frac{n!(n-1)!!}{(n/2)!}2^{n/2} & \text{if} & n\text{ is even}%
\end{array}%
\right. .
\end{equation*}

vii) Taking $\rho \allowbreak =\allowbreak 1$ in (\ref{Ex2}) we get for all
complex $\beta \neq 0$ and $n\geq 0$:%
\begin{equation*}
\sum_{m=0}^{n}\left( -1\right) ^{m}\binom{n}{m}\left( \beta \right) ^{\left(
m\right) }\left( \beta \right) ^{\left( n-m\right) }\sum_{k=0}^{m}\binom{m}{k%
}\binom{n-m}{k}\frac{k!}{\left( \beta \right) ^{\left( k\right) }}=0.
\end{equation*}
\end{remark}

\begin{remark}
i) Taking $\rho \allowbreak =\allowbreak 0$ in (\ref{Ex3}) we get 
\begin{equation*}
\sum_{m=0}^{n}(-1)^{n-m}\binom{n}{m}\left( \beta \right) ^{\left( n-m\right)
}\left( \beta \right) ^{\left( m\right) }=\left\{ 
\begin{array}{ccc}
0 & \text{if} & n\text{ is odd} \\ 
\frac{n!}{(n/2)!}\left( \beta \right) ^{\left( n/2\right) } & \text{if} & n%
\text{ is even}%
\end{array}%
\right. ,
\end{equation*}%
which is related, but not a direct generalization of a well-known (also from
umbral calculus) equality:%
\begin{equation*}
\sum_{m=0}^{n}\binom{n}{m}\left( \beta \right) ^{\left( n-m\right) }\left(
\alpha \right) ^{\left( m\right) }=\left( \alpha +\beta \right) ^{\left(
n\right) }.
\end{equation*}

ii) Taking $\rho \allowbreak =\allowbreak 1/2$ in (\ref{Ex3}), we get for
all complex $\beta \neq 0:$%
\begin{gather*}
\sum_{m=0}^{n}(-1)^{n-m}\binom{n}{m}2^{-m}\sum_{j=0}^{m}\binom{m}{j}\left(
\beta \right) ^{\left( n-j\right) }\left( \beta +m-j\right) ^{\left(
j\right) } \\
=\left\{ 
\begin{array}{ccc}
0 & \text{if} & n\text{ is odd} \\ 
\frac{n!}{2^{n/2}(n/2)!}\left( \beta \right) ^{\left( n/2\right) } & \text{if%
} & n\text{ is even}%
\end{array}%
\right. .
\end{gather*}

iii) Changing the order of summation in (\ref{Ex3}), comparing the
coefficients in expansions in powers of $\rho $, multiplying both sides by $%
\left( \beta \right) ^{\left( k\right) }(n-k)!$ and cancelling out $n!$ we
get for all $\beta $ complex and $0\leq k\leq n$:%
\begin{eqnarray*}
&&\sum_{m=k}^{n}(-1)^{m-k}\frac{(n-k)!}{(m-k)!(n-m-k)!}\left( \beta \right)
^{\left( m\right) }\left( \beta \right) ^{\left( n-m\right) } \\
&=&\left\{ 
\begin{array}{ccc}
0 & \text{if} & n\text{ is odd } \\ 
(n/2)!\binom{n-k}{n/2}\left( \beta \right) ^{\left( k\right) }\left( \beta
\right) ^{\left( n/2\right) } & \text{if} & n\text{ is even}%
\end{array}%
\right. .
\end{eqnarray*}
\end{remark}

We have also the following relationships between Fibonacci $\left\{
F_{n}\right\} _{n\geq 0}$ and Lucas $\left\{ L_{n}\right\} _{n\geq 0}$
numbers. For Fibonacci and Lucas numbers see, e.g.: \cite{Ball03}, \cite%
{Knuth98} or \cite{Cull05}. For other relationships of Fibonacci and Lucas
numbers with the so-called Catalan triangles (expressed by binomial
coefficients) see, e.g., \cite{Szabl21}.

\begin{corollary}
Let us denote by $\phi $ the following number $\left( \sqrt{5}-1\right) /2$
that is called "golden ratio". Take $\rho \allowbreak =\allowbreak -\phi $
in (\ref{G1}). We get then 
\begin{eqnarray*}
2\frac{\left( 2k\right) !}{k!}L_{k} &=&\sum_{j=0}^{k}(-1)^{j}\binom{2k}{2j}%
(2j-1)!!(2k-2j-1)!!(L_{2k-2j}L_{j}-5F_{2k-2j}F_{j}), \\
2\frac{\left( 2k\right) !}{k!}F_{k} &=&\sum_{j=0}^{k}(-1)^{j}\binom{2k}{2j}%
(2j-1)!!(2k-2j-1)!!(F_{2k-2j}L_{j}-L_{2k-2j}F_{j}).
\end{eqnarray*}
\end{corollary}

\begin{proof}
We know that for all $n\geq 0:$ 
\begin{equation*}
(1+\phi )^{n}=((1+\sqrt{5})/2)^{n}=L_{n}/2+F_{n}\sqrt{5}/2.
\end{equation*}
Now notice also that we have $\phi \allowbreak =\allowbreak 1/(1+\phi )$ and
that 
\begin{equation*}
L_{n}^{2}-5F_{n}^{2}=(-1)^{n}4.
\end{equation*}%
Hence we have: 
\begin{equation*}
\phi ^{n}=(-1)^{n}(L_{n}/2-F_{n}\sqrt{5}/2).
\end{equation*}%
Now we take $\rho \allowbreak =\allowbreak -\phi $ and insert it in (\ref{G1}%
). Notice that $1-\rho \allowbreak =\allowbreak (1+\phi )$, hence $(1-\rho
)^{2k-2j}\allowbreak =\allowbreak L_{2k-2j}/2\allowbreak +\allowbreak
F_{2k-2j}\sqrt{5}/2$ and $1-\rho ^{2}\allowbreak =\allowbreak \phi $,
consequently $(1-\rho ^{2})^{j}\allowbreak =\allowbreak
(-1)^{j}(L_{j}/2\allowbreak -\allowbreak F_{j}\sqrt{5}/2)$. Now it remains
to perform multiplication and separation of terms with and without $\sqrt{5}$%
.
\end{proof}

\section{Probabilistic background and the proofs}

\subsection{Probabilistic background}

All identities presented in Theorem \ref{Main} stem from calculating the
following moments $E(X-Y)^{n},$ $n\geq 0,$ where $X$ and $Y$ are two
normalized random variables. The joint distribution of these random
variables has one parameter and is either bivariate Normal (Gaussian) or
bivariate Gamma distribution. More precisely, in the first case, the joint
distribution of $(X,Y)$ has the following well-known density, valid for all $%
x,y\in \mathbb{R}$ and $\left\vert \rho \right\vert <1:$ 
\begin{equation}
f_{N}(x,y|\rho )=\frac{1}{2\pi (1-\rho ^{2})}\exp \left( -\frac{x^{2}-2\rho
xy+y^{2}}{2(1-\rho ^{2})}\right) ,  \label{2xN}
\end{equation}
while in the second case valid for all $x,y,\beta >0$ and $\left\vert \rho
\right\vert <1:$ 
\begin{equation*}
f_{G}(x,y|\rho )=f_{g}(x|\beta )f_{g}(y|\beta )\frac{\exp \left( -\frac{\rho
\left( x+y\right) }{1-\rho }\right) }{\left( 1-\rho \right) \left( xy\rho
\right) ^{(\beta -1)/2}}I_{\beta -1}\left( \frac{2\sqrt{xy\rho }}{1-\rho }%
\right) ,
\end{equation*}%
where 
\begin{equation*}
f_{g}(x|\beta )\allowbreak =\allowbreak x^{\beta -1}\exp (-x)/\Gamma (\beta
),
\end{equation*}
for $x>0$ and $0$ otherwise. $I_{\alpha }$ denotes the modified Bessel
function of the first kind. The most important property of these joint
densities is that they allow the so-called Lancaster expansions. To learn
more about Lancaster expansions, their probabilistic interpretations and
convergence problems associated with them, see, e.g., \cite{Alexits61}, \cite%
{Lancaster58}, \cite{Lancaster75}, \cite{Szeg75}, \cite{Szabl21}, \cite%
{SzabStac}. The orthogonal polynomials that we are using are well presented,
for example, in two well-known monographs \cite{Chih79}, and \cite{Szeg75}.

Namely, in the first case, we have the so-called Poisson-Mehler expansion 
\begin{equation}
f_{N}(x,y|\rho )=f_{N}(x)f_{N}(x)\sum_{n\geq 0}\rho ^{n}h_{n}(x)h_{n}(y),
\label{MN}
\end{equation}%
where we denoted for simplicity 
\begin{equation*}
f_{N}(x)\allowbreak =\allowbreak \exp (-x^{2}/2)/\sqrt{2\pi },
\end{equation*}%
and 
\begin{equation*}
h_{n}(x)\allowbreak =\allowbreak H_{n}(x)/\sqrt{n!}.
\end{equation*}%
$h_{n}(x)$ is the orthonormal modification of the so-called probabilistic
Hermite polynomials $\left\{ H_{n}(x)\right\} $, i.e. polynomials defined by
the following three-term recurrence :%
\begin{equation*}
H_{n+1}(x)\allowbreak =\allowbreak xH_{n}(x)-nH_{n-1}(x),
\end{equation*}%
with $H_{-1}(x)\allowbreak =\allowbreak 0$ and $H_{0}(x)\allowbreak
=\allowbreak 1.$ It is also known, that 
\begin{equation}
\frac{1}{\sqrt{2\pi }}\int_{-\infty }^{\infty }h_{n}(x)h_{m}(x)\exp
(-x^{2}/2)dx=\delta _{nm},  \label{ortH}
\end{equation}%
where $\delta _{nm}$ denotes Kronecker's delta.

In the second case we have 
\begin{equation}
f_{G}(x,y|\rho )=f_{g}(x|\beta )f_{g}(y|\beta )\sum_{j\geq 0}\rho
^{n}l_{n}(x|\beta )l_{n}(y|\beta ),  \label{MG}
\end{equation}%
where $l_{n}(x|\beta )\allowbreak =\allowbreak \sqrt{\frac{n!}{\left( \beta
\right) ^{\left( n\right) }}}L_{n}(x|\beta )$ is the orthonormal
modification of the so-called generalized Laguerre polynomials $\left\{
L_{n}(x|\beta )\right\} $, defined by the following expansions: 
\begin{equation}
L_{n}(x|\beta )=\sum_{k=0}^{n}(-1)^{k}\frac{(\beta )^{(n)}}{(n-k)!(\beta
)^{(k)}}x^{k}/k!.  \label{defL}
\end{equation}

Let us remark that the expansion (\ref{MG}) is known under the name
Hardy-Hille (see, e.g. \cite{Szeg75}, p.102) formula. One knows that : 
\begin{equation}
\int_{0}^{\infty }l_{n}(x|\beta )l_{m}(x|\beta )f_{g}(x|\beta )dx=\delta
_{n,m}.  \label{ortL}
\end{equation}

Our aim is to calculate $\left\{ E(X-Y)^{n}\right\} _{n\geq 0}.$ We will do
it in two ways.

The first way is to calculate the generating function of the set of these
numbers i.e. the function%
\begin{equation*}
g(t)=\sum_{n\geq 0}t^{n}E(X-Y)^{n}/n!.
\end{equation*}%
It can be calculated in the following way since it is known that all moments
exist, hence one can exchange integration and summation. Namely, we have%
\begin{equation*}
g(t)=\int \int \exp \left( tx-ty\right) f(x,y|\rho )dxdy.
\end{equation*}%
Obviously here and below the integration is depending on the case either
over whole $\mathbb{R}^{2}$ or over $\mathbb{R}^{+}\times \mathbb{R}^{+}$.
The second way is to calculate the number using an expansion: 
\begin{equation*}
E(-X+Y)^{n}=\sum_{m=0}^{n}(-1)^{m}EX^{m}Y^{n-m}.
\end{equation*}%
Now 
\begin{equation*}
EX^{m}Y^{n-m}=\int \int x^{m}y^{n-m}f(x,y|\rho )dxdy.
\end{equation*}%
where $f$ is either $f_{N}$ or $f_{G}$ presented above. Using one of the
expansions (\ref{MN}) or (\ref{MG}) we will find these moments using the
numbers 
\begin{equation*}
H_{m,j}=EX^{m}k_{j}(X),
\end{equation*}%
where $k_{i}(x)$ is either $h_{j}(x)$ if we consider the Normal case or $%
l_{j}(x)$ if we consider Gamma case. Since both expansions (\ref{MN}) and (%
\ref{MG}) have similar structure we have:%
\begin{equation}
EX^{m}Y^{n-m}=\sum_{j=0}^{\min (m,n-m)}\rho ^{j}H_{m,j}H_{n-m,j}.
\label{ExpH}
\end{equation}%
This is so since $X$ and $Y$ have the same distributions and since $%
H_{m,j}\allowbreak =\allowbreak 0$ if $j>m$ because of the orthogonality of
polynomials $h$ or $l.$ We will calculate this function assuming either
expansion (\ref{MN}) or (\ref{MG}).

\begin{case}
\label{Gaussian}So let us start with expansion (\ref{MN}). First let us
calculate the auxiliary numbers $H_{j,n}$. We have: 
\begin{equation}
H_{j,n}=EX^{j}h_{n}(X)=\frac{1}{\sqrt{n!}}EX^{j}H_{n}(X)=\left\{ 
\begin{array}{ccc}
0 & \text{if } & n>j\text{ or }j-n\text{ odd} \\ 
\frac{j!}{2^{(j-n)/2}((j-n)/2)!\sqrt{n!}} & \text{if} & j-n\text{ is even}%
\end{array}%
\text{.}\right.  \label{HjnN}
\end{equation}%
This is so since polynomials $h_{n}$ are orthogonal and also since it is
common knowledge, that $\forall n\geq 0:$%
\begin{equation*}
x^{j}\allowbreak =\allowbreak j!\sum_{m=0}^{\left\lfloor j/2\right\rfloor }%
\frac{1}{2^{m}m!(j-2m)!}H_{j-2m}(x).
\end{equation*}%
Secondly, let us calculate the following auxiliary function that will
simplify many further calculations.

\begin{eqnarray*}
m_{n}(t)\allowbreak &=&\allowbreak Eh_{n}(X)\exp (tX)\allowbreak
=\allowbreak \frac{1}{\sqrt{2\pi n!}}\int_{-\infty }^{\infty }H_{n}(x)\exp
(tx)\exp (-x^{2}/2)dx \\
&=&\frac{\exp (t^{2}/2)}{\sqrt{2\pi n!}}\int_{-\infty }^{\infty
}H_{n}(x)\exp (-(x-t)^{2}/2)dx.
\end{eqnarray*}%
Now, let us change the variable under the integral by setting $y\allowbreak
=\allowbreak (x-t).$ Then, we get:%
\begin{equation*}
m_{n}(t)\allowbreak =\allowbreak \frac{\exp (t^{2}/2)}{\sqrt{2\pi n!}}%
\int_{-\infty }^{\infty }H_{n}(y+t)\exp (-y^{2}/2)dy.
\end{equation*}%
Next, we utilize the following the well-known expansion%
\begin{equation*}
H_{n}(x+y)\allowbreak =\allowbreak \sum_{j=0}^{n}\binom{n}{j}H_{j}(x)y^{n-j}.
\end{equation*}%
Now, since we have (\ref{ortH}), we see that $m_{n}(t)\allowbreak
=\allowbreak t^{n}\frac{\exp (t^{2}/2)}{\sqrt{n!}}$ and we get: 
\begin{gather*}
E\exp \left( tX-tY\right) \allowbreak =\allowbreak \sum_{n\geq 0}E\exp
\left( tX-tY\right) \\
=\frac{1}{2\pi }\sum_{n\geq 0}\rho ^{n}\int_{-\infty }^{\infty
}\int_{-\infty }^{\infty }\exp \left( tx-ty\right) \exp \left( -\left(
x^{2}+y^{2}\right) /2\right) h_{n}(x)h_{n}(y)dxdy \\
\sum_{n\geq 0}\rho ^{n}m_{n}(t)m_{n}(-t)=\exp (t^{2})\sum_{n\geq 0}(-\rho
)^{n}t^{2n}/n!=\exp (t^{2}(1-\rho )) \\
\sum_{n\geq 0}t^{2n}(1-\rho )^{n}/n!.
\end{gather*}%
Thus, we deduce that: 
\begin{equation}
E(X-Y)^{n}=\left\{ 
\begin{array}{ccc}
0 & \text{if } & n\text{ is odd} \\ 
\frac{n!}{(n/2)!}(1-\rho )^{n/2} & \text{if } & n\text{ is even}%
\end{array}%
\right. .  \label{nG}
\end{equation}
\end{case}

\begin{case}
\label{Gamma}Now let us consider expansion (\ref{MG}). We need to recall
some simple, well-known facts: The Gamma distribution with rate parameter
zero and shape parameter $\beta >0,$ is the distribution with the density $%
f_{g}$ presented above. It is easy to see, recalling the definition of the
Euler's Gamma function that: 
\begin{equation}
EX^{n}=\int_{0}^{\infty }x^{n}f_{g}(x|\beta )dx=\left( \beta \right)
^{\left( n\right) }.  \label{momG}
\end{equation}%
As before, let us calculate the set of auxiliary quantities and functions:
We start with the numbers: 
\begin{equation*}
H_{j,n}=EX^{j}l_{n}(X).
\end{equation*}%
We have:%
\begin{equation}
H_{j,n}=\frac{\sqrt{n!}}{\Gamma (\beta )\sqrt{\left( \beta \right) ^{\left(
n\right) }}}\int_{0}^{\infty }x^{j}L_{n}^{(\beta )}(x)x^{\beta -1}\exp
(-x)dx.  \label{defH}
\end{equation}%
Now, we use the well-known expansion of $x^{j}$ in powers of Laguerre
polynomials:%
\begin{equation}
x^{j}\allowbreak =\allowbreak j!\sum_{k=0}^{j}(-1)^{k}\frac{(\beta )^{(j)}}{%
(j-k)!(\beta )^{(k)}}L_{k}^{(\beta )}(x),  \label{xL}
\end{equation}%
and then we use the orthogonality of the Laguerre polynomials. Hence we have:%
\begin{equation}
H_{j,n}=(-1)^{n}\binom{j}{n}(\beta )^{(j)}\sqrt{\frac{n!}{\left( \beta
\right) ^{\left( n\right) }}}.  \label{HjnG}
\end{equation}

Let us calculate now, the following auxiliary functions:%
\begin{eqnarray*}
m_{n}(t) &=&El_{n}(X)\exp (tX)=\sqrt{\frac{n!}{\left( \beta \right) ^{\left(
n\right) }}}E\exp (tX)L_{n}(X|\beta )\allowbreak \\
&=&\allowbreak \sqrt{\frac{n!}{\left( \beta \right) ^{\left( n\right) }}}%
\frac{1}{\Gamma (\beta )}\int_{0}^{\infty }\exp (xt)L_{n}(x|\beta )x^{\beta
-1}\exp (-x)dx\allowbreak \\
&=&\allowbreak \sqrt{\frac{n!}{\left( \beta \right) ^{\left( n\right) }}}%
\frac{1}{\Gamma (\beta )}\int_{0}^{\infty }L_{n}(x|\beta )x^{\beta -1}\exp
(-x(1-t))dx\allowbreak .
\end{eqnarray*}%
Now, let us change variables under the integral by considering $y\allowbreak
=\allowbreak x(1-t).$ Then we get: $\allowbreak $%
\begin{equation*}
m_{n}(t)=\sqrt{\frac{n!}{\left( \beta \right) ^{\left( n\right) }}}\frac{1}{%
1-t}\frac{1}{\Gamma (\beta )}\int_{0}^{\infty }L_{n}(\frac{y}{1-t}|\beta
)\left( \frac{y}{1-t}\right) ^{\beta -1}\exp (-y)dy\allowbreak .
\end{equation*}%
We will now utilize the following well-known recurrence relation for
Laguerre polynomials (see, e.g., \cite{Abram64}) :%
\begin{equation*}
L_{n}(y|\beta )=\sum_{j=0}^{n}L_{n-j}(x|\beta +j)(y-x)^{j}/j!.
\end{equation*}%
So we get: 
\begin{gather*}
m_{n}(t)=\allowbreak \sqrt{\frac{n!}{\left( \beta \right) ^{\left( n\right) }%
}}\frac{1}{(1-t)^{\beta }}\frac{1}{\Gamma (\beta )}\times \\
\int_{0}^{\infty }\left( \sum_{j=0}^{n}\left( \frac{-ty}{1-t}\right)
^{j}L_{n-j}(x|(\beta +j))/j!\right) y^{\beta -1}\exp (-y)dy.
\end{gather*}%
Further, we have: 
\begin{gather*}
m_{n}(t)=\allowbreak \allowbreak \sqrt{\frac{n!}{\left( \beta \right)
^{\left( n\right) }}}\frac{1}{(1-t)^{\beta }}\frac{1}{\Gamma (\beta )}\times
\\
\sum_{j=0}^{n}\int_{0}^{\infty }(-1)^{j}\frac{t^{j}}{(1-t)^{j}}%
L_{n-j}(y\left\vert \left( \beta +j\right) \right\vert )y^{j+\beta -1}\exp
(-y)dy\allowbreak \\
=\allowbreak (-1)^{n}\sqrt{\frac{n!}{\left( \beta \right) ^{\left( n\right) }%
}}\frac{t^{n}\Gamma (\beta +n)}{\Gamma (\beta )n!(1-t)^{\beta +n}}%
\allowbreak =\allowbreak (-1)^{n}\sqrt{\frac{n!}{\left( \beta \right)
^{\left( n\right) }}}\frac{t^{n}\left( \beta \right) ^{\left( n\right) }}{%
n!(1-t)^{\beta +n}}.\allowbreak
\end{gather*}%
Hence we have now:%
\begin{gather*}
\allowbreak g(t)=E\exp (tX-tY)=\sum_{n\geq 0}\rho ^{n}\int_{0}^{\infty }\exp
(tx-ty)l_{n}(x|\beta )l_{n}(y|\beta )f_{g}(x|\beta )f_{g}(y|\beta
)dxdy\allowbreak \allowbreak \\
=\allowbreak \sum_{n\geq 0}\rho ^{n}m_{n}(t)m_{n}(-t)=\allowbreak
\allowbreak \sum_{n\geq 0}\rho ^{n}\frac{n!(-1)^{n}t^{2n}\left( \beta
\right) ^{\left( n\right) }\left( \beta \right) ^{\left( n\right) }}{\left(
\beta \right) ^{\left( n\right) }n!n!(1-t^{2})^{n+\beta }} \\
=\frac{1}{(1-t^{2})^{\beta }}\sum_{n\geq 0}\rho ^{n}\frac{%
(-1)^{n}t^{2n}\left( \beta \right) ^{\left( n\right) }}{n!(1-t^{2})^{n}} \\
=\frac{1}{(1-t^{2})^{\beta }}\sum_{n\geq 0}(-1)^{n}\left( \frac{\rho t^{2}}{%
1-t^{2}}\right) ^{n}\left( \beta \right) ^{\left( n\right) }/n!,
\end{gather*}%
and further we have:%
\begin{gather*}
g(t)\allowbreak =\allowbreak \frac{1}{(1-t^{2})^{\beta }}\left( 1+\frac{\rho
t^{2}}{1-t^{2}}\right) ^{-\beta }\allowbreak = \\
=\allowbreak \frac{1}{(1-t^{2}(1-\rho ))^{\beta }}=\sum_{n\geq
0}t^{2n}(1-\rho )^{n}\left( \beta \right) ^{\left( n\right) }/n!.
\end{gather*}%
We use here twice the well-known formula for the so-called binomial series.
Thus we deduce 
\begin{equation}
E(X-Y)^{n}\allowbreak =\allowbreak \left\{ 
\begin{array}{ccc}
0 & \text{if } & n\text{ is odd} \\ 
\frac{n!}{(n/2)!}\left( \beta \right) ^{\left( n/2\right) }(1-\rho )^{n/2} & 
\text{if } & n\text{ is even}%
\end{array}%
\right. .  \label{EXYG}
\end{equation}%
$\allowbreak $
\end{case}

\subsection{Proofs}

All proofs are based on considering different ways of calculating numbers $%
\left\{ E(X-Y)^{n}\right\} _{n\geq 1}$ when assuming one of the discussed
above joint distributions.

Proof of assertion i) of Theorem \ref{Main}:

Let $(X,Y)$ have bivariate Gaussian distribution given by (\ref{2xN}). Then
it is well known that every linear transformation of such vriable has also
gaussin distribution. In particular we deduce that $X-Y\sim N(0,(1-\rho
)^{2}).$ hence, keeping in mind that $EZ^{j}\allowbreak =\allowbreak \left\{ 
\begin{array}{ccc}
0 & \text{if} & j\text{ is odd} \\ 
1 & \text{if} & j=0 \\ 
\sigma ^{j}(2j-1)!! & \text{if} & j\text{ is even}%
\end{array}%
\right. $ if only $Z\sim N(0,\sigma ^{2})$. Thus the left hand side of
assertion i) is equal to $(2k-1)!!(1-\rho )^{k}\allowbreak =\allowbreak 
\frac{(2k)!}{k!}(1-\rho )^{k}.$ Now it is enough to change $\rho $ to $-\rho 
$.

In order to get right hand side we calculate: 
\begin{gather*}
E(X-Y)^{2k}\allowbreak =\allowbreak E(X+\rho Y-(1+\rho )Y)^{2k}\allowbreak =
\\
\allowbreak \sum_{j=0}^{2k}\binom{2k}{j}(-1)^{j}(1+\rho )^{2k-j}E(X+\rho
Y)^{j}Y^{2k-j}\allowbreak = \\
\allowbreak \sum_{j=0}^{2k}\binom{2k}{j}(-1)^{j}(1+\rho )^{2k-j}E(E(X+\rho
Y)^{j}|Y)Y^{2k-j}\allowbreak .
\end{gather*}%
Now, we use the fact know that random variables $(X+\rho Y)$ and $Y$ are
independent, hence $(E(X+\rho Y)^{j}|Y)\allowbreak =\allowbreak \left\{ 
\begin{array}{ccc}
0 & \text{if} & j\text{ is odd} \\ 
1 & \text{if} & j=0 \\ 
(2j-1)!!(1-\rho ^{2})^{j/2} & \text{if} & j\text{ is even}%
\end{array}%
\right. .$ Thus we get%
\begin{equation*}
E(X-Y)^{2k}=\sum_{j=0}^{k}\binom{2k}{2j}(1-\rho )^{2k-2j}(1-\rho
^{2})^{j}(2j-1)!!(2k-2j-1)!!.
\end{equation*}

Proof of assertion ii) of Theorem \ref{Main}:

We use (\ref{ExpH}) with (\ref{HjnN}) getting:

\begin{equation*}
\frac{(2k)!}{k!}(1-\rho )^{k}=2\sum_{j=0}^{k-1}(-1)^{j}\binom{2k}{j}%
\sum_{n=0}^{j}\frac{\rho ^{n}}{n!}H_{j,n}H_{2k-j,n}+(-1)^{k}\binom{2k}{k}%
\sum_{n=0}^{k}\frac{\rho ^{n}}{n!}H_{k,n}^{2},
\end{equation*}%
where $H_{j,n}=\left\{ 
\begin{array}{ccc}
0 & \text{if } & n>j\text{ or }j-n\text{ odd} \\ 
\frac{j!}{2^{(j-n)/2}((j-n)/2)!} & \text{if} & j-n\text{ is even}%
\end{array}%
\text{.}\right. $ Notice that if $j-n$ is even then $\binom{2k}{j}%
H_{j,n}H_{2k-j,n}\allowbreak =\allowbreak 2^{n}\binom{2k}{k}\binom{k}{n}%
\binom{k-n}{(j-n)/2}/2^{k}$ and when $k-n$ is even $H_{k,n}^{2}\allowbreak
=\allowbreak 2^{n}\binom{k}{n}\binom{k-n}{(k-n)/2}/2^{k},$ hence (\ref{G2})
can be reformulated in the following way:%
\begin{equation*}
\frac{1}{2^{k-1}}\sum_{j=0}^{k-1}(-1)^{j}\sum_{n=0-}^{j}(2\rho )^{n}\binom{k%
}{n}mh_{j,n,k}+\frac{(-1)^{k}}{2^{k}}\sum_{n=0}^{k}(2\rho )^{n}\binom{k}{n}%
mh_{k,n,k},=(1-\rho )^{k},
\end{equation*}%
where $mh_{j,n,k}\allowbreak =\allowbreak \left\{ 
\begin{array}{ccc}
0 & \text{if} & j>k-1\text{ or }j-n\text{ is odd} \\ 
\binom{k-n}{(j-n)/2} & \text{if } & j-n\text{ is even}%
\end{array}%
.\right. $ Now it remains to change the order of summation in the internal
sums and $\rho $ to $-\rho $.

Proof of assertion iii) of Theorem \ref{Main}:

Let us recall (\ref{defL}) and (\ref{defH}) and let us calculate $H_{j,n}$
directly, getting:%
\begin{gather*}
H_{j,n}=\frac{\sqrt{n!}}{\Gamma (\beta )\sqrt{\Gamma (n+\beta )}}%
\sum_{k=0}^{n}\frac{(-1)^{k}}{k!}\frac{(\beta )^{(n)}}{(n-k)!(\beta )^{(k)}}%
\int_{0}^{\infty }x^{k+j}x^{\beta -1}\exp (-x)dx \\
=\frac{\sqrt{n!}}{\sqrt{\Gamma (n+\beta )}}\sum_{k=0}^{n}\frac{(-1)^{k}}{k!}%
\frac{(\beta )^{(n)}(\beta )^{(k+j)}}{(n-k)!(\beta )^{(k)}}.
\end{gather*}%
Now, we compare it with the calculated already $H_{j,n}$ for Gamma
distribution i.e., $(-1)^{n}\binom{j}{n}(\beta )^{(j)}\sqrt{\frac{n!}{\left(
\beta \right) ^{\left( n\right) }}}.$

Proof of assertion iv) of Theorem \ref{Main}:

Recall (\ref{HjnG}). We have:%
\begin{gather*}
\left\{ 
\begin{array}{ccc}
0 & \text{if } & n\text{ is odd} \\ 
\frac{n!}{(n/2)!}\left( \beta \right) ^{\left( n/2\right) }(1-\rho )^{n/2} & 
\text{if } & n\text{ is even}%
\end{array}%
\right. =E(-X+Y)^{n}\allowbreak \\
=\sum_{m=0}^{n}(-1)^{m}\binom{n}{m}EX^{m}Y^{n-m}.\allowbreak
\end{gather*}%
Now we use (\ref{ExpH}) with $H_{j,m}$ given by (\ref{HjnG}). By (\ref{EXYG}%
) we have 
\begin{gather*}
E(-X+Y)^{n}=\sum_{m=0}^{n}(-1)^{m}\binom{n}{m}\sum_{j=0}^{\min (m,n-m)}\rho
^{j}H_{m,j}H_{n-m,j} \\
=\sum_{m=0}^{n}(-1)^{n-m}\binom{n}{m}\sum_{j=0}^{\min (m,n-m)}\rho ^{j}%
\binom{m}{j}\binom{n-m}{j}\left( \beta \right) ^{\left( m\right) }\left(
\beta \right) ^{\left( n-m\right) }\frac{j!}{\left( \beta \right) ^{\left(
j\right) }} \\
=\sum_{m=0}^{n}(-1)^{n-m}\binom{n}{m}\left( \beta \right) ^{\left( m\right)
}\left( \beta \right) ^{\left( n-m\right) }\sum_{j=0}^{m}\binom{m}{j}\binom{%
n-m}{j}\frac{j!\rho ^{j}}{\left( \beta \right) ^{\left( j\right) }} \\
=\left\{ 
\begin{array}{ccc}
0 & \text{if } & n\text{ is odd} \\ 
\frac{n!}{(n/2)!}\left( \beta \right) ^{\left( n/2\right) }(1-\rho )^{n/2} & 
\text{if } & n\text{ is even}%
\end{array}%
\right. .
\end{gather*}%
This is so, since $\binom{m}{j}\binom{n-m}{j}$ is zero whenever $m>j$ or $%
n-m>j.$

Proof of assertion iv) of Theorem \ref{Main}:

One can easily notice, basing on (\ref{MG}), that $E(l_{n}(X)|Y)\allowbreak
=\allowbreak \rho ^{n}l_{n}(Y).$ Let us hence calculate first, the
conditional moments $\eta _{n}(y|\beta ,\rho )\allowbreak =\allowbreak
E(X^{n}|Y=y)$. We have by (\ref{xL}) 
\begin{eqnarray*}
\eta _{j}(y|\beta ,\rho ) &=&E(X^{j}|Y=y)\allowbreak =\allowbreak
j!\sum_{k=0}^{j}(-1)^{k}\frac{(\beta )^{(j)}}{(j-k)!(\beta )^{(k)}}\rho
^{k}L_{k}^{(\beta )}(y) \\
&=&j!\sum_{k=0}^{j}(-1)^{k}\frac{(\beta )^{(j)}}{(j-k)!(\beta )^{(k)}}\rho
^{k}\sum_{m=0}^{k}(-1)^{m}\frac{(\beta )^{(k)}}{(k-m)!(\beta )^{(m)}}y^{m}/m!
\\
&=&j!\sum_{m=0}^{j}\frac{(\beta )^{(j)}}{(\beta )^{(m)}(j-m)!m!}(\rho
y)^{m}\sum_{k=m}^{j}(-1)^{k-m}\frac{(j-m)!}{(k-m)!(j-k)!}\rho ^{k-m} \\
&=&\sum_{m=0}^{j}\binom{j}{m}\frac{(\beta )^{(j)}}{(\beta )^{(m)}}(\rho
y)^{m}(1-\rho )^{j-m}=j!(1-\rho )^{j}L_{j}(-\frac{\rho y}{1-\rho }).
\end{eqnarray*}

Further we get:%
\begin{gather*}
E\left( (X-Y)^{n}|Y=y\right) =\sum_{j=0}^{n}(-1)^{n-j}\binom{n}{j}\eta
_{j}(y|\beta ,\rho )y^{n-j} \\
=\sum_{j=0}^{n}(-1)^{n-j}y^{n-j}\binom{n}{j}\sum_{m=0}^{j}\binom{j}{m}\frac{%
(\beta )^{(j)}}{(\beta )^{(m)}}(\rho y)^{m}(1-\rho )^{j-m} \\
=\sum_{m=0}^{n}\frac{n!}{m!(n-m)!}\frac{(\rho y)^{m}}{(\beta )^{(m)}}%
\sum_{j=m}^{n}(-1)^{n-j}\frac{(n-m)!(\beta )^{(j)}y^{n-j}(1-\rho )^{j-m}}{%
(n-j)!(j-m)!} \\
\sum_{m=0}^{n}\frac{n!}{m!(n-m)!}(\rho y)^{m}\sum_{s=0}^{n-m}(-1)^{n-m-s}%
\binom{n-m}{s}y^{n-m-s}(1-\rho )^{s}(\beta +m)^{(s)} \\
=\sum_{t=0}^{n}(-1)^{n-t}y^{n-t}\frac{n!}{t!(n-t)!}\sum_{s=0}^{t}\binom{t}{s}%
(\rho y)^{t-s}(1-\rho )^{s}(\beta +t-s)^{(s)}
\end{gather*}%
\begin{gather*}
=\sum_{s=0}^{n}\binom{n}{s}y^{n-s}(1-\rho )^{s}\sum_{t=s}^{n}(-1)^{n-t}\frac{%
(n-s)!}{(t-s)!t!}\rho ^{t-s}(\beta +t-s)^{(s)} \\
=\sum_{s=0}^{n}\binom{n}{s}y^{n-s}(1-\rho )^{s}\sum_{m=0}^{n-s}(-1)^{n-m-s}%
\binom{n-s}{m}\rho ^{m}(\beta +m)^{(s)}.
\end{gather*}

Hence we get using (\ref{momG}) we have; 
\begin{gather*}
E(X-Y)^{n}=\int_{0}^{\infty }(\sum_{s=0}^{n}\binom{n}{s}y^{n-s}(1-\rho
)^{s}\times \\
\sum_{m=0}^{n-s}(-1)^{n-m-s}\binom{n-s}{m}\rho ^{m}(\beta
+m)^{(s)})f_{g}(x|\beta )dx \\
=\sum_{s=0}^{n}\binom{n}{s}(\beta )^{(n-s)}(1-\rho
)^{s}\sum_{m=0}^{n-s}(-1)^{n-m-s}\binom{n-s}{m}\rho ^{m}(\beta +m)^{(s)} \\
=\sum_{t=0}^{n}(-1)^{n-t}\frac{n!(\beta )^{(n-t)}}{t!(n-t)!}\sum_{s=0}^{t}%
\frac{t!}{s!(t-s)!}(1-\rho )^{s}\rho ^{t-s}(\beta +n-t)^{(t-s)}(\beta
+t-s)^{(s)}.
\end{gather*}

In the last line, we have changed the order of summation. To get the
assertion we use the formula: 
\begin{equation*}
\left( \beta \right) ^{\left( n\right) }\left( \beta +n\right) ^{\left(
m\right) }=\left( \beta \right) ^{\left( n+m\right) },
\end{equation*}%
true for all complex $\beta $.


\begin{thebibliography}{99}
\bibitem{Alexits61} Alexits, G. Convergence problems of orthogonal series.
Translated from the German by I. F\"{o}lder. \emph{International Series of
Monographs in Pure and Applied Mathematics}, Vol. \textbf{20} Pergamon
Press, New York-Oxford-Paris 1961 \{\TEXTsymbol{\backslash}rm ix\}+350 pp.
MR0218827

\bibitem{Abram64} Abramowitz, Milton; Stegun, Irene Ann, eds. (1983) [June
1964]. "Chapter 22". Handbook of Mathematical Functions with Formulas,
Graphs, and Mathematical Tables. Applied Mathematics Series. Vol. 55 (Ninth
reprint with additional corrections of tenth original printing with
corrections (December 1972); first ed.). Washington D.C.; New York: United
States Department of Commerce, National Bureau of Standards; Dover
Publications. p. 773

\bibitem{Ball03} Ball, Keith M (2003), "8: Fibonacci's Rabbits Revisited",
Strange Curves, Counting Rabbits, and Other Mathematical Explorations,
Princeton, NJ: \emph{Princeton University Press}

\bibitem{Cull05} Cull, Paul; Flahive, Mary; Robson, Robby. Difference
equations. From rabbits to chaos. Undergraduate Texts in Mathematics. \emph{%
Springer}, New York, 2005. xiv+392 pp. ISBN: 0-387-23233-8 MR2131908

\bibitem{Chih79} Chihara, T. S. An introduction to orthogonal polynomials.
Mathematics and its Applications, Vol. 13. \emph{Gordon and Breach Science
Publishers, New York-London-Paris, }1978. xii+249 pp. ISBN: 0-677-04150-0
MR0481884 (58 \#1979)

\bibitem{Grad14} Gradshteyn, I. S.; Ryzhik, I. M. (2014). Table of
Integrals, Series, and Products (8th ed.). \emph{Academic Press}.

\bibitem{Knuth98} Graham, Ronald L.; Knuth, Donald E.; Patashnik, Oren
(1994). Concrete Mathematics (Second ed.). \emph{Addison-Wesley}. pp.
153--256.

\bibitem{Lancaster58} H. O. Lancaster, The structure of bivariate
distributions, \emph{Ann. Math. Statistics}, vol. \textbf{29}, no. 3, pp.
719-736, September 1958.

\bibitem{Lancaster75} Lancaster, H. O. Joint probability distributions in
the Meixner classes. \emph{J. Roy. Statist. Soc. Ser. B }\textbf{37} (1975),
no. 3, 434--443. MR0394971 (52 \#15770)

\bibitem{SzabKes} Szab\l owski, Pawe\l\ J. Yet another way of calculating
moments of the Kesten's distribution and its consequences for Catalan
numbers and Catalan triangles, \emph{Discrete Math.} \textbf{345} (2022),
no. 9, Paper No. 112891. MR4417241

\bibitem{SzabStac} Szab\l owski, Pawe\l\ J. On stationary Markov processes
with polynomial conditional moments. \emph{Stoch. Anal. Appl.} \textbf{35}
(2017), no. 5, 852--872. MR3686472

\bibitem{Szabl21} Pawe\l\ J. Szab\l owski, On positivity of orthogonal
series and its applications in probability, \emph{Positivity }\textbf{26, }%
article 19(2022\textbf{)}, https://arxiv.org/abs/2011.02710.

\bibitem{Szeg75} G. Szeg\H{o}, Orthogonal polynomials, 4th edition, \emph{%
Amer. Math. Soc. Colloq. Publ.}, vol. 23, Amer. Math. Soc., Providence, RI,
1975
\end{thebibliography}
\end{document}